\documentclass[12pt]{article}

\usepackage{amsmath, amsthm, amssymb}

\usepackage{fullpage}

\usepackage{hyperref}

\newtheorem{theorem}{Theorem}
\newtheorem{lemma}{Lemma}

\newtheorem{conjecture}{Conjecture}
\newtheorem{corollary}{Corollary}
\theoremstyle{definition}
\newtheorem{question}{Question}

\newtheorem{definition}{Definition}
\newtheorem{remark}{Remark}

\newcommand{\cG}{\mathcal{G}}
\newcommand{\cH}{\mathcal{H}}

\newcommand{\cA}{\mathcal{A}}

\newcommand{\cF}{\mathcal{F}}

\newcommand{\cM}{\mathcal{M}}
\newcommand{\cT}{\mathcal{T}}

\title{Short proofs of three results about intersecting systems}
\author{J\'ozsef Balogh\thanks{Department of Mathematical Sciences, University of Illinois at Urbana-Champaign, IL, USA. Email: \texttt{jobal@illinois.edu}. Research supported by NSF RTG Grant DMS-1937241, NSF Grant DMS-1764123 and Arnold O. Beckman Research Award (UIUC) Campus Research Board 18132, the Langan Scholar Fund (UIUC)  and the Simons Fellowship.}  
\and
William Linz\thanks{Department of Mathematical Sciences, University of Illinois at Urbana-Champaign, IL, USA. Email: \texttt{wlinz2@illinois.edu}. Partially supported by the Hohn-Nash Fellowship (UIUC) and NSF RTG Grant DMS-1937241.}}
\date{\today}

\begin{document}
\maketitle

\begin{abstract}
In this note, we give short proofs of three theorems about intersection problems. The first one is a determination of the maximum size of a nontrivial $k$-uniform, $d$-wise intersecting family for $n\ge \left(1+\frac{d}{2}\right)(k-d+2)$, which improves the range of $n$ a recent result of O'Neill and Verstra\"{e}te. Our proof also extends to $d$-wise, $t$-intersecting families, and from this result we obtain a version of the Erd\H{o}s-Ko-Rado theorem for $d$-wise, $t$-intersecting families.

Our second result partially proves a conjecture of Frankl and Tokushige about $k$-uniform families with restricted pairwise intersection sizes. 

Our third result is about intersecting families of graphs. Answering a question of Ellis, we construct $K_{s, t}$-intersecting families of graphs which have size larger than the Erd\H{o}s-Ko-Rado-type construction, whenever $t$ is sufficiently large in terms of $s$. The construction is based on nontrivial $(2s)$-wise $t$-intersecting families of sets.

\end{abstract}

\section{Introduction}

Let $\cF \subseteq 2^{[n]}$ be a family of subsets of the set $[n]:= \{1, 2, \ldots, n\}$. The family $\cF$ is \textit{$d$-wise $t$-intersecting} if for every $F_1, \ldots, F_d \in \cF$, we have $|\bigcap_{i=1}^dF_i| \ge t$. The family $\cF$ is \textit{nontrivial ($d$-wise) $t$-intersecting} if $\cF$ is ($d$-wise) $t$-intersecting and  $|\bigcap_{F\in \cF} F| < t$. If $\cF \subset \binom{[n]}{k}$ and $L \subset [0, k-1]$, then $\cF$ is an \emph{$(n, k, L)$-system} if $|F\cap F'| \in L$ for all distinct $F, F'\in \cF$. 

Let $\cG \subseteq 2^{\binom{[n]}{2}}$ be a family of labelled graphs on $n$ vertices. Given a fixed, unlabelled graph $H$, the family $\cG$ is \emph{$H$-intersecting} if for every pair of graphs $G_1, G_2 \in \cG$, the graph $G_1\cap G_2$ contains a copy of $H$ (note that $(V(G_1\cap G_2) = V(G_1) = V(G_2)$ and $E(G_1\cap G_2) = E(G_1) \cap E(G_2)$). For a graph $H$, the \emph{Erd\H{o}s-Ko-Rado-type} (shortened hereafter to EKR-type) construction of an $H$-intersecting family of graphs is the family $\cH$ of graphs on $n$ vertices which contain a fixed copy of $H$. Notice that when $\cH$ is EKR-type, then $|\cH| = 2^{\binom{n}{2}-|E(H)|}$. If the maximum-size $H$-intersecting family of graphs is EKR-type, then we say the graph $H$ has the \emph{EKR property}. 

In this note, we give short proofs of three results about set systems satisfying some intersection properties and about intersecting families of graphs. First, we determine the maximum size of a nontrivial $d$-wise intersecting $k$-uniform family for $n \ge \left(1+\frac{d}{2}\right)(k-d+2)$, which greatly extends the range of a recent result of O'Neill and Verstra\"{e}te \cite{OV}. Our result also extends to nontrivial $d$-wise, $t$-intersecting, $k$-uniform families, and from this we derive a version of the Erd\H{o}s-Ko-Rado theorem for $d$-wise, $t$-intersecting families. 

Our second theorem partially settles a conjecture of Frankl and Tokushige \cite{FT} for a general bound for the maximum size of an $(n, k, L)$-system when $L = [0, \ell - 1] \cup [\ell+1, k-1]$.  

Our third result is a construction of $K_{s, t}$-intersecting families whose size is larger than the EKR-type construction for $K_{s, t}$, answering a question of Ellis~\cite{E}. The construction is based on nontrivial $(2s)$-wise $t$-intersecting families of sets.

\subsection{Nontrivial intersecting families}

We define a pair of nontrivial $d$-wise intersecting families $\cA(n, k, d)$ and $\cH(n, k, d)$. 
\begin{definition}\label{defn:akdhkddefn}

\[\cA(n, k,d) := \left\{A\in \binom{[n]}{k}: |A\cap [d+1]| \ge d\right\}.\]
\[\cH(n, k,d) := \left\{A\in \binom{[n]}{k}: [d-1] \subset A, A\cap [d, k+1] \neq \emptyset\right\}\bigcup \bigg\{[k+1]\setminus {i}: i\in [d-1]\bigg\}.\]

\end{definition}
Hilton and Milner \cite{HM} conjectured that one of the two families $\cA(n, k, d)$ or $\cH(n, k, d)$ is the maximum-size $k$-uniform $d$-wise intersecting family, for $n$ sufficiently large. 

In the case $d=2$, the classical result of Hilton and Milner \cite{HM} determines the maximum size of a nontrivial intersecting $k$-uniform family. 

\begin{theorem}[Hilton-Milner \cite{HM}]\label{thm:hm}
Let $\cF\subset \binom{[n]}{k}$ be a nontrivial intersecting family. For $n > 2k$, we have
\[|\cF| \le \binom{n-1}{k-1} - \binom{n-k-1}{k-1} + 1.\]

Furthermore, equality holds for $k > 3$ only if $\cF \cong \cH(n, k, 2)$. If $k=3$, then equality holds only if $\cF \cong \cH(n, 3, 2)$ or $\cF \cong \cA(n, 3, 2)$. 
\end{theorem}

For large $n$, O'Neill and Verstra\"{e}te \cite{OV} recently verified the conjecture of Hilton and Milner~\cite{HM} for nontrivial $d$-wise intersecting families. 

\begin{theorem}[O'Neill-Verstra\"{e}te \cite{OV}]\label{thm:ov}Let $k$ and $d$ be integers with $2\le d < k$. Then, there is an $n_0 = n_0(k, d)$ such that for every $n\ge n_0$, if $\cF \subset \binom{[n]}{k}$ is a nontrivial $d$-wise intersecting family, then we have
\[|\cF|\le \max\left\{|\cH(n, k, d)|, |\cA(n, k,d)|\right\}.\]
\end{theorem}

O'Neill and Verstra\"{e}te \cite{OV} showed that one can choose $n_0(k, d) = d+e(k^22^k)^{2^k}(k-d)$. If $n < kd/(d-1)$, then every $\cF \subset \binom{[n]}{k}$ is $d$-wise intersecting. O'Neill and Verstra\"{e}te~\cite{OV} conjectured that Theorem \ref{thm:ov} holds for $n \ge kd/(d-1)$. 

\begin{conjecture}[O'Neill-Verstra\"{e}te \cite{OV}]\label{conj:ov} In Theorem \ref{thm:ov}, one can choose \[n_0(k, d) = \frac{kd}{d-1}.\]
\end{conjecture}

We greatly extend the range of $n$ for which Theorem \ref{thm:ov} is known, and also give a counterexample to the range conjectured in Conjecture \ref{conj:ov} when $n \sim kd/(d-1)$.

\begin{theorem}\label{thm:dwiseintthm}
If $\cF \subseteq \binom{[n]}{k}$ is a nontrivial $d$-wise intersecting family with $2\le d < k$, then for every $n > \left(1+\frac{d}{2}\right)(k-d+2)$, we have
\[|\cF| \le \max\left\{|\cH(n, k, d)|, |\cA(n, k, d)|\right\}.\]

Furthermore, equality holds only if $\cF$ is isomorphic to one of $\cA(n, k, d)$ or $\cH(n, k, d)$. 
\end{theorem}

In fact, our argument readily extends to provide a bound for the maximum size of a nontrivial $d$-wise $t$-intersecting family. 

\begin{theorem}\label{thm:dwisetintthm}
If $\cF\subseteq \binom{[n]}{k}$ is a nontrivial $d$-wise $t$-intersecting family with $k\ge t+d-1$, $d\ge 2$ and $t\ge 1$, then for $n > \left(\frac{t+d+1}{2}\right)(k-t-d+3)$, we have
\[|\cF| \le \max\left\{|\cH(n, k, t+d-1)|, |\cA(n, k, t+d-1)|\right\}.\]

Furthermore, equality holds only if $\cF$ is isomorphic to either $\cA(n, k, t+d-1)$ or $\cH(n, k, t+d-1)$. 
\end{theorem}

The case $d=2$ of Theorem \ref{thm:dwisetintthm} is part of the complete nontrivial intersection theorem of Ahlswede and Khachatrian~\cite{AK2}, which is a key part of our proofs of Theorems \ref{thm:dwiseintthm} and \ref{thm:dwisetintthm}. We state part of their result, for $n > (t+1)(k-t+1)$. 

\begin{theorem}[Ahlswede-Khachatrian \cite{AK2}]\label{thm:akntcithm}
Let $\cF \subset \binom{[n]}{k}$ be a nontrivial $t$-intersecting family. Then, for every $n > (t+1)(k-t+1)$, the following holds:

(i) If $k\le 2t+1$, then \[|\cF| \le |\cA(n, k, t+1)|,\]
and $\cA(n, k, t+1)$ is up to isomorphism the unique optimal family. 

(ii) If $k > 2t+1$, then \[|\cF| \le \max\{|\cA(n, k, t+1)|, |\cH(n, k, t+1)|\},\]
and these are the only optimal families, up to isomorphism. 
\end{theorem}

We also use one case of the complete intersection theorem of Ahlswede and Khachatrian~\cite{AK1, AK3} in the proof of Theorem~\ref{thm:dwisetintthm}. 

\begin{theorem}[Ahlswede-Khachatrian \cite{AK1, AK3}]\label{thm:aktcithm}
Let $\cF \subset \binom{[n]}{k}$ be a $t$-intersecting family with $k\ge t$. Then, for $(k-t+1)(2+\frac{t-1}{2}) < n < (k-t-1)(t+1)$, we have 
\[|\cF| \le |\cA(n, k, t+1)|,\]
and $\cA(n, k, t+1)$ is the unique optimal family, up to isomorphism.
\end{theorem}

Theorem \ref{thm:dwisetintthm} allows us to conclude a version of the Erd\H{o}s-Ko-Rado theorem~\cite{EKR} for $d$-wise $t$-intersecting families.

\begin{theorem}\label{thm:ekrdwisetintthm}
If $\cF\subseteq \binom{[n]}{k}$ is a $d$-wise $t$-intersecting family with $k\ge t$, $d\ge 2$ and $t\ge 1$, then for $n > (t+d-1)(k-t-d+3)$, we have
\[|\cF| \le \binom{n-t}{k-t},\]
and equality holds only if $\cF$ is isomorphic to $\{F \in \binom{[n]}{k}:\ [t] \subseteq F\}$.
\end{theorem}

Note that the Erd\H{o}s-Ko-Rado theorem is essentially the case $d=2$, $t=1$ of Theorem \ref{thm:ekrdwisetintthm}. No version of the Erd\H{o}s-Ko-Rado theorem for $d$-wise $t$-intersecting families with such an explicit range as $n \ge Ck$ (where $C$ depends only on $t$ and $d$) has previously appeared in the literature.  There have been a number of results that are of a more asymptotic nature \cite{T} or consider specific values of $d$ or $t$. Tokushige~\cite{Tok3} proved that the maximum size of a $3$-wise $t$-intersecting family is $\binom{n-t}{k-t}$ for $t\ge 26$, $n$ sufficiently large and the optimal range of roughly $n\ge \sqrt{t}k$. 

The range for which Theorem~\ref{thm:ekrdwisetintthm} holds is certainly not optimal unless $d=2$. We can show that the conclusion of Theorem~\ref{thm:ekrdwisetintthm} holds for $n \ge ck$, where $c = c_{t, d} < t+d-1$ is a slightly better constant (depending on $t$ and $d$) and $n$ is sufficiently large. 

In order to state the result, we define the polynomial $f_{t, d}(x) = (1-x)^{t+d-3} - x^{d-2}$ for $t\ge 2$ and $d\ge 3$. Let $\beta_{t, d} \in (0, \frac12)$ be a root of $f_{t, d}$, so that $\beta_{t, d}$ satisfies
\begin{equation}\label{eqn:betaeqn}
(1-\beta_{t, d})^{t+d-3} - \beta_{t, d}^{d-2} = 0.
\end{equation}

It is easy to check that $\frac{1}{t+1} < \beta_{t, d} < \frac12$ for $d\ge 3$ and $t\ge 2$.  

\begin{theorem}\label{thm:improvedekrdwisetintthm}
If $\cF\subseteq \binom{[n]}{k}$ is a $d$-wise $t$-intersecting family with $k\ge t$, $d\ge 3$ and $t\ge 2$, then for any constant $c$ with $c > 1/\beta_{t, d}$, there exists an $n_0$ such that if $n \ge n_0$ and $n \ge c k$, we have
\[|\cF| \le \binom{n-t}{k-t},\]
and equality holds only if $\cF$ is isomorphic to $\{F \in \binom{[n]}{k}:\ [t] \subseteq F\}$.
\end{theorem}

For a set $F\in 2^{[n]}$, for $0 < p < 1$, we define the product measure $\mu_p(F) = p^{|F|}(1-p)^{n - |F|}$, and for a family $\cF \subset 2^{[n]}$, we define $\mu_p(\cF) = \sum_{F\in \cF}\mu_p(F)$. We also obtain a version of Theorem \ref{thm:ekrdwisetintthm} for $d$-wise, $t$-intersecting families for the measure $\mu_p$ from Theorem~\ref{thm:improvedekrdwisetintthm}. 

\begin{theorem}\label{thm:pekrdwisetintthm}
Let $\cF \subset 2^{[n]}$ be a $d$-wise, $t$-intersecting family. Then, for $0 < p < \beta_{t, d}$, 
\[\mu_p(\cF) \le p^t.\]
\end{theorem}

Theorem \ref{thm:pekrdwisetintthm} is a corollary of a general phenomenon whereby ``discrete" $k$-uniform results like Theorem \ref{thm:improvedekrdwisetintthm} can often be boosted to ``smooth" product measure results like Theorem \ref{thm:pekrdwisetintthm} (see, for instance, \cite{Tok1}). 

\subsection{Families missing one intersection}
In \cite[pg.~215]{FT}, Frankl and Tokushige made the following conjecture.

\begin{conjecture}\label{conj:missingoneintconj}
Suppose $\cF \subset \binom{[n]}{k}$ satisfies $|F \cap F'| \neq \ell$ for some $2\ell < k$, for every pair of sets $F, F' \in \cF$. Then, 
\[|\cF| \le \binom{n}{k-\ell-1}.\]
\end{conjecture}

It is noted in \cite{FT} that this conjecture is true if $k-\ell$ is a prime power. Also, if Conjecture~\ref{conj:missingoneintconj} is true, then it is tight for $k=2\ell+1$ and infinitely many values of $n$; see Frankl \cite{F83}.

Our contribution is to note that for each fixed $\ell$, Conjecture \ref{conj:missingoneintconj} is true for all but at most a finite number of values of $k$. 

\begin{theorem}\label{thm:extensionthm}
Suppose $\cF \subset \binom{[n]}{k}$ satisfies $|F \cap F'| \neq \ell$ for some $2\ell < k$, for every pair of sets $F, F' \in \cF$. If $k-\ell$ does not divide $\ell!$, then 
\[|\cF| \le \binom{n}{k-\ell-1}.\]
In particular, this holds when $\ell < \log k/\log\log k$, where $\log$ is the natural logarithm. 
\end{theorem}

\subsection{H-intersecting families}
Ellis~\cite{E} asked which graphs $H$ have the EKR property, and in particular whether $H$ has the EKR property whenever $H$ is $2$-connected. Ellis, Filmus and Friedgut~\cite{EFF} proved that the graph $H=K_3$ has the EKR property, and Berger and Zhao~\cite{BZ} recently proved the same result for  $H=K_4$. On the other hand, Christofides~\cite{C} gave a construction of an $H$-intersecting family of larger size than the EKR-type construction when $H=P_3$, and the EKR-type construction is also not optimal for disjoint unions of stars (unless $H=K_2$)~\cite{AS}. 

We give a construction of $K_{s, t}$-intersecting families which have size larger than the EKR-type construction for $K_{s, t}$, whenever $t$ is sufficiently large in terms of $s$. In particular, for every $s\ge 1$, there are $s$-connected graphs $H$ which do not have the EKR property, answering Ellis's question. 

\begin{theorem}\label{thm:kstintthm}
Let $s$ and $t$ be positive integers with $t > 2^{2s} - 2s - 1$. Then there exists a $K_{s, t}$-intersecting family of graphs $\cF$ on $n$ vertices with $|\cF| > 2^{\binom{n}{2} - st}$.
\end{theorem}

As all previously known examples of graphs without the EKR property are bipartite, a natural question is: are there examples of graphs $H$ with $\chi(H) > 2$ which do not have the EKR property? We can modify the construction in Theorem~\ref{thm:kstintthm} to give graphs with arbitrary large connectivity and arbitrarily large chromatic number which do not have the EKR property. 

\begin{theorem}\label{ks1srtint}
Let $s_1, s_2, \ldots, s_r, t$ be integers with $s_i \ge 1$ for $1\le i\le r$ and $t > 2^{2\sum_{i} s_i} -2\sum_{i}s_i - 1$. Then there exists a $K_{s_1, \ldots, s_r, t}$-intersecting family of graphs $\cH$ with \[|\cH| > 2^{\binom{n}{2} - \sum_{1\le i < j \le r}s_is_j - \sum_{i=1}^rs_i t}.\]
\end{theorem}

\subsection{Organization of the paper} 
Our proofs of Theorems \ref{thm:dwiseintthm} and \ref{thm:dwisetintthm} are presented in Section 2. The proofs of the EKR-type results Theorems \ref{thm:ekrdwisetintthm} and \ref{thm:pekrdwisetintthm} are presented in Section 3. The counterexample to Conjecture~\ref{conj:ov} is described in Section 4. The proof of Theorem \ref{thm:extensionthm} is given in Section 5. The proof of Theorem~\ref{thm:kstintthm} appears in Section 6. 

\section{Proofs of Theorems \ref{thm:dwiseintthm} and \ref{thm:dwisetintthm}}

We begin with a simple observation about $m$-wise intersections in nontrivial $d$-wise, $t$-intersecting families. This observation was noted in \cite{OV} in the case $t=1$, and was proven at least as early as \cite{FT3}.

\begin{lemma}\label{lem:msetslem}
Suppose $\cF$ is a nontrivial $d$-wise $t$-intersecting family. Let $A_1, \ldots, A_m \in \cF$, where $m\le d$. Then, 
\[|\cap_{i=1}^m A_i| \ge t + d - m.\]
\end{lemma}
\begin{proof}
Suppose that $\bigcap_{A\in\cF}A = \{x_1, \ldots, x_c\}$, where $c\le t-1$, since $\cF$ is a nontrivial intersecting family. Suppose that $\bigcap_{i=1}^mA_i = \{x_1, \ldots, x_c\} \cup \{y_1, \ldots, y_{\ell}\}$, and assume for a contradiction that $c + \ell \le t+d-m-1$. For each element $y_i$, there is a set $B_i \in \cF$ such that $y_i \notin B_i$.  If $\ell \le d - m$, then $|(\cap_{i=1}^{\ell}B_i) \bigcap (\cap_{i=1}^m A_i )|  = c \le t-1$, which contradicts the $d$-wise $t$-intersecting property. On the other hand, if $\ell > d - m$, then $|(\cap_{i=1}^{d-m}B_i) \bigcap (\cap_{i=1}^m A_i )|  \le c + \ell - (d-m)  \le t-1$, which again contradicts the $d$-wise $t$-intersecting property. 
\end{proof}

By setting $m=2$ in Lemma \ref{lem:msetslem}, we immediately obtain the following corollary. 
\begin{corollary}\label{cor:dwisecor}
If $\cF$ is a nontrivial $d$-wise $t$-intersecting family, then $\cF$ is also a $(t+d-2)$-intersecting family. 
\end{corollary}

We will only prove Theorem \ref{thm:dwisetintthm}, as Theorem \ref{thm:dwiseintthm} follows from Theorem \ref{thm:dwisetintthm} by setting $t=1$. Theorem \ref{thm:dwisetintthm} is a  straightforward consequence of Theorems~\ref{thm:akntcithm} and \ref{thm:aktcithm} and Corollary \ref{cor:dwisecor}.

\begin{proof}[Proof of Theorem \ref{thm:dwisetintthm}]
If $\cF \subset \binom{[n]}{k}$ is a nontrivial $d$-wise $t$-intersecting family, then $|\cap_{F\in \cF}F| \le t - 1 < t + d - 2$, so by Corollary \ref{cor:dwisecor}, $\cF$ is also a nontrivial $(t+d-2)$-intersecting family. Theorem \ref{thm:akntcithm} now immediately implies Theorem \ref{thm:dwisetintthm} when $n\ge (t+d-1)(k- t - d + 3)$.  If $\left(\frac{t+d+1}{2}\right)(k-t-d+3) < n < (t+d-1)(k-t-d+3)$, then Theorem \ref{thm:aktcithm} immediately implies the result.
\end{proof}

\begin{remark}\label{rem:AKCITimprov}
For $n > 2k - t - d +2$, the complete intersection theorem of Ahlswede and Khachatrian~\cite{AK1, AK3} also provides an upper bound on the maximum size of a nontrivial $d$-wise $t$-intersecting family. However, the optimal $(t+d-2)$-intersecting families for $n < \left(\frac{t+d+1}{2}\right)(k-t-d+3)$ in the complete intersection theorem are not nontrivial $d$-wise $t$-intersecting.
\end{remark}

O'Neill and Verstra\"{e}te~\cite[Theorem~2]{OV} additionally proved a stability result for nontrivial $d$-wise intersecting families and sufficiently large values of $n$. Corollary~\ref{cor:dwisecor} does not seem to immediately imply their stability result. It would be interesting to obtain a similar stability result for nontrivial $d$-wise $t$-intersecting families as a further extension of Theorem~\ref{thm:dwisetintthm}. In the case $d=2$ and $t=1$, Han and Kohayakawa~\cite{HK} (for all values of $n$) and Kostochka and Mubayi~\cite{KM} (for sufficiently large values of $n$) independently obtained stability results for the Hilton-Milner theorem by determining the maximum size of a nontrivial intersecting family that is not a subfamily of $\cH(n, k, 2)$. 

Note that Lemma \ref{lem:msetslem} and Corollary \ref{cor:dwisecor} make no assumption on the sizes of the sets in $\cF$, so one might wonder if these could be useful tools for the nonuniform case as well. Katona \cite{K} determined the maximum size of a $t$-intersecting family $\cF \subset 2^{[n]}$. In general, the extremal families $\cF$ are nontrivial and not $d$-wise, $(t-d+1)$-intersecting, so one cannot hope to obtain a proof like that of Theorem \ref{thm:dwisetintthm}. On the other hand, Frankl \cite{F} obtained a number of results for nonuniform nontrivial $d$-wise, $t$-intersecting families, and one of the tools used is a version of Corollary \ref{cor:dwisecor} (see also Section 4 of the survey of Frankl and Tokushige \cite{FT1}). 

Let us also mention that the nonuniform version of Theorem~\ref{thm:dwisetintthm} in the case $t=1$ was already proven by Brace and Daykin~\cite{BD}. Brace and Daykin proved that if $\cF \subset 2^{[n]}$ is a nontrivial $d$-wise intersecting family, then $|\cF| \le |\cA(n, d)|$, where $\cA(n, d):= \{A\subseteq [n] : |A\cap [d+1]|\ge d\}$ is the nonuniform analogue of $\cA(n, k, d)$. 

\section{The uniform and measure Erd\H{o}s-Ko-Rado theorems}

We now give a simple proof of Theorem \ref{thm:ekrdwisetintthm} as a corollary of Corollary~\ref{cor:dwisecor} and the Erd\H{o}s-Ko-Rado theorem for $t$-intersecting families. 

\begin{theorem}[Erd\H{o}s-Ko-Rado for $t$-intersecting families]\label{thm:ekrtintthm}
Let $\cF \subset \binom{[n]}{k}$ be a $t$-intersecting family. Then, for $n \ge (t+1)(k-t+1)$,
\[|\cF| \le \binom{n-t}{k-t},\]
and for $n > (t+1)(k-t+1)$, equality holds if and only if $\cF$ is isomorphic to $\{F \in \binom{[n]}{k}:\ [t] \subseteq F\}$.
\end{theorem}

Theorem~\ref{thm:ekrtintthm} was proven by Frankl~\cite{F3} for $t\ge 15$ and by Wilson~\cite{Wil} for all $t$. 

\begin{proof}[Proof of Theorem~\ref{thm:ekrdwisetintthm}]
Let $\cF \subset \binom{[n]}{k}$ be a nontrivial $d$-wise, $t$-intersecting family. By Corollary~\ref{cor:dwisecor}, $\cF$ is also $(t+d-2)$-intersecting. Hence, by Theorem~\ref{thm:ekrtintthm}, for $n > (t+d-2+1)(k-(t+d-2)+1) = (t+d-1)(k-t-d+3)$, we have
\[|\cF| \le \binom{n-t-d+2}{k-t-d+2} \le \binom{n-t}{k-t}.\]
If $d\ge 3$, then the second inequality is strict and so equality holds if and only if $\cF$ is trivial. 
\end{proof}

Note that Theorem~\ref{thm:ekrdwisetintthm} is interesting if $t+d-1 \le k < 2t+d-2$, as then $(t-d+1)(k-t-d+3) \le (t+1)(k-t+1)$. For $k \ge 2t+d-2$, Theorem~\ref{thm:ekrdwisetintthm} can be deduced from Theorem~\ref{thm:ekrtintthm} and the fact that $d$-wise $t$-intersecting families are also ($2$-wise) $t$-intersecting families. 

In order to improve the range for which Theorem~\ref{thm:ekrdwisetintthm} holds, we need to use a good upper bound for the size of a $t$-intersecting family for smaller values of $n$.  For simplicity, we use a recent upper bound of Frankl~\cite{F1}. 

\begin{theorem}[Frankl]\label{thm:n-1k-tthm}
Let $\cF \subset \binom{[n]}{k}$ be a $t$-intersecting family. Then, if $n \ge 2k-t+1$, 
\[|\cF|\le \binom{n-1}{k-t}.\]
\end{theorem}

We also need the following lemma for fixed values of $k$/$n$. 

\begin{lemma}\label{pbetalem}
Let $t$ and $d$ be integers with $t\ge 2$ and $d\ge 3$. Let $p$ be a rational number with $0 < p < \beta_{t, d}$, where $\beta_{t, d}$ is defined as in \eqref{eqn:betaeqn}. Then, there exists $n_0$ such that if $n > n_0$ and $\frac{k}{n} = p$, we have
\[\binom{n-t}{k-t} > \binom{n-1}{k-t-d+2}.\]
\end{lemma}

\begin{proof}[Proof of Lemma~\ref{pbetalem}]
Expanding both binomial coefficients, we need to show that for $n > n_0$, 
\[\frac{(n-t)(n-t+1)\cdot\ldots\cdot(n-k+1)}{(k-t)!} > \frac{(n-1)(n-2)\cdot\ldots\cdot(n-k+t+d-2)}{(k-t-d+2)!}.\]
Rearranging, this is equivalent to 
\begin{equation}\label{eqn:improvedbound}
(n-k+t+d-3)\cdot\ldots\cdot (n-k+1) > (n-1)\cdot\ldots\cdot(n-t+1)(k-t)\cdot\ldots\cdot(k-t-d+3).
\end{equation}
We now set $k=pn$. Both sides of \eqref{eqn:improvedbound} are polynomials in $n$ of degree $t+d-3$. The coefficient of $n^{t+d-3}$ on the left-hand side is $(1-p)^{t+d-3}$, while the coefficient of $n^{t+d-3}$ on the right-hand side is $p^{d-2}$. From the definition of $\beta_{t, d}$ in \eqref{eqn:betaeqn}, it follows that $(1-p)^{t+d-3} > p^{d-2}$ whenever $0 < p < \beta_{t, d}$. Hence, the conclusion of Lemma~\ref{pbetalem} holds if $n$ is sufficiently large. 
\end{proof}

Theorem~\ref{thm:improvedekrdwisetintthm} is now an immediate consequence of Corollary~\ref{cor:dwisecor}, Theorem~\ref{thm:n-1k-tthm} and Lemma~\ref{pbetalem}. 

It is possible that further improvements on the range of the constant $c$ in Theorem~\ref{thm:improvedekrdwisetintthm} could be obtained by using the full statement of the complete intersection theorem of Ahlswede and Khachatrian~\cite{AK1, AK3} in place of Theorem~\ref{thm:n-1k-tthm}. However, the proof for such improvements may be quite technical. 

To obtain Theorem~\ref{thm:pekrdwisetintthm}, we use the following theorem of Tokushige~\cite[Theorem~1]{Tok1} relating results about the maximum size of $k$-uniform $d$-wise $t$-intersecting families and the maximum product measure of nonuniform $d$-wise $t$-intersecting families. 

\begin{theorem}[Tokushige]\label{thm:tokthm}
Let $d\ge 2$ and $t\ge 1$ be integers, and let $p \in (0, 1)$. Then, statement (i) implies statement (ii).

(i) Let $\cF \subset \binom{[n]}{k}$ be a $d$-wise $t$-intersecting family. Then there exist $\epsilon$ and $n_0$ such that the inequality
\[|\cF| \le \binom{n-t}{k-t}\]
holds whenever $n > n_0$ and $|\frac{k}{n} - p| \le \epsilon$.

(ii) Let $\cF \subset 2^{[n]}$ be a $d$-wise $t$-intersecting family. Then,
\[\mu_p(\cF) \le p^t\]
holds for all $n\ge t$. 
\end{theorem}

Theorem~\ref{thm:pekrdwisetintthm} can be immediately deduced from Theorems~\ref{thm:improvedekrdwisetintthm} and \ref{thm:tokthm}. 

It would be interesting to prove Theorem~\ref{thm:pekrdwisetintthm} without the use of Theorem~\ref{thm:improvedekrdwisetintthm}. In the case of $t$-intersecting families, Friedgut~\cite{Fried} obtained the product measure version of Theorem~\ref{thm:ekrtintthm} by using a version of the Hoffman bound. Friedgut also obtained uniqueness and stability statements and it would be interesting to have analogous statements for Theorem~\ref{thm:pekrdwisetintthm}. 

\section{Counterexamples to Conjecture \ref{conj:ov}}

We briefly return to nontrivial $d$-wise intersecting families, which initially motivated our research. It is not difficult to show that $|\cA(n, k, d)| \ge |\cH(n, k, d)|$ when $n\le d(k-d+2)$, so Conjecture~\ref{conj:ov} reduces to asking if $\cA(n, k, d)$ is the maximum size nontrivial $d$-wise intersecting family for this range of values of $n$. We present some counterexamples to this conjecture. 

Let $n=11$ and $k=7$, and consider the nontrivial $3$-wise intersecting family 
\[\cM(11, 7, 3, 3):= \{F \in \binom{[11]}{7}: 1\in F, \hspace{2mm} |F\cap [2, 8]|\ge 4\} \cup \{[2, 8]\}.\]

It is easy to verify that $|\cM(11, 7, 3, 3)| = 176$, while $\max\{|\cA(11, 7, 3)|, |\cH(11, 7, 3)|\} = 175$. This family was originally discovered as the solution to an integer program solved by the Gurobi optimization software \cite{G}; see Wagner \cite{W} for an illuminating discussion of this method. 

In general, for $1\le r \le \left\lfloor{\frac{k-1}{d-1}}\right\rfloor$, define
\[\cM(n, k, d, r):= \left\{F \in \binom{[n]}{k}: 1\in F, |F \cap [2, 2+(d-1)r]| \ge 1+(d-2)r\right\} \] \[\bigcup \left\{F \in\binom{[2, n]}{k}: F \supseteq [2, 2+(d-1)r]\right\}. \]
It is straightforward to verify that these families are nontrivial $d$-wise intersecting. Now, define 
\[m(n, k, d):= \max_{1\le r\le \lfloor{\frac{k-1}{d-1}\rfloor}} |\cM(n, k, d, r)|. \]

Note $\cA(n, k, d) \cong \cM(n, k, d, 1)$. We conjecture that $m(n, k, d)$ is the maximum size of a nontrivial $d$-wise intersecting family for $n$ in this range of values.

\begin{conjecture}\label{conj:dwiseintconj}
If $\cF \subseteq \binom{[n]}{k}$ is a nontrivial $d$-wise intersecting family, then for $k\frac{d}{d-1}\le n \le \left(1+\frac{d}{2}\right)(k-d+2)$, we have 
\[|\cF| \le m(n, k, d).\]
Furthermore, all extremal families are isomorphic to some $\cM(n, k, d, r)$. 
\end{conjecture}

Potentially, an appropriate adaptation of the proof of Ahlswede and Khachatrian \cite{AK2} could be used to prove Conjecture \ref{conj:dwiseintconj}. We managed to slightly extend the range for $n$ over Theorem~\ref{thm:dwiseintthm}, but could not prove the entire conjecture.

We do not have a precise guess as to which value of $r$ gives the maximal $|\cM(n, k, d, r)|$, nor for what values of $n$ we should expect $m(n, k, d) > |\cA(n, k, d)|$. Computations indicate that $\cA(n, k, d)$ is maximal in the range $n \ge k\frac{d-1}{d-2}$, while if $n \sim k \frac{d}{d-1}$ and $k$ is sufficiently large in terms of (fixed) $d$, then $\cA(n, k, d)$ will not be maximal. We formulate a question for further study.

\begin{question}\label{ques:dwiseintques}
Assuming Conjecture \ref{conj:dwiseintconj} is true, which $\cM(n, k, d, r)$ would give the maximum size family? When would $\cA(n, k, d)$ be the maximum size family?
\end{question}

It would be interesting to know if this construction extends further to the case when $t > 1$.
\begin{question}\label{ques:nontrivdwisetint}
Let $d\ge 3$ and $t\ge 2$. Are there further examples of nontrivial $d$-wise $t$-intersecting families $\cF \subset \binom{[n]}{k}$ with
\[\binom{n-t}{k-t} > |\cF| > \max\{|\cA(n, k, t+d-1)|, |\cH(n, k, t+d-1)|\}\]
for some $n > kd/(d-1) - t$?
\end{question}

We have added the hypothesis $|\cF| < \binom{n-t}{k-t}$ to Question~\ref{ques:nontrivdwisetint} because such constructions would be most interesting in cases where the maximum-size $k$-uniform $d$-wise $t$-intersecting family is trivial. 

\begin{remark}\label{rem:tpkushige}
Subsequent to our Conjecture~\ref{conj:dwiseintconj}, Tokushige~\cite{Tok2} constructed a collection of nontrivial $d$-wise intersecting families that can be counterexamples to Conjecture~\ref{conj:dwiseintconj}. Namely, for $0\le r\le d-1$, he defined the families 

\[\cT(n, k, d, r):= \left\{A\in \binom{[n]}{k}: [r] \in A, |A\cap [r+1, k+1]| \ge j_0\right\} \bigcup \left\{[k+1]\setminus \{i\}: 1\le i\le r\right\},\]

where $j_0 = \lfloor{\frac{d-r-1}{d-r}(k-r+1)\rfloor}+1$. 
Note that $\cT(n, k, d, d-1) \cong \cH(n, k, d)$. Similarly to the definition for $\cM$, define
\[t(n, k, d):= \max_{0\le r\le d-1}|\cT(n, k, d, r)|.\]
One can check that, for example, $t(120, 77, 4) > m(120, 77, 4)$. 

Tokushige~\cite[Problem~2]{Tok2} asks the following question:
\begin{question}[Tokushige]\label{ques:tokushige}
Let $\cF \subset \binom{[n]}{k}$ be a nontrivial $d$-wise intersecting family. Is it true that for $n \ge kd/(d-1)$,
\[|\cF| \le \max\{|\cA(n, k, d)|, t(n, k, d)\}?\]
\end{question}

The answer to Question~\ref{ques:tokushige} as stated is ``no", because there are examples where $m(n, k, d) > t(n, k, d)$ and $m(n, k, d)> |\cA(n, k, d)|$. The smallest example is $n=12$, $k=8$, and $d=3$, where $m(12, 8, 3) = |\cM(12, 8, 3, 3)| = 299 > 261 = t(12, 8, 3)$. We therefore propose a modified version of Tokushige's question. 

\begin{question}\label{ques:linz}
Let $\cF \subset \binom{[n]}{k}$ be a nontrivial $d$-wise intersecting family. Is it true that for $n \ge kd/(d-1)$,
\[|\cF| \le \max\{m(n, k, d), t(n, k, d)\}?\]
\end{question}

One reason to think that the answer to Question~\ref{ques:linz} might be ``yes" is that $\cM(n, k, d, r)$ and $\cT(n, k, d, r)$ are collections of families which include and generalize the families $\cA(n, k, d)$ and $\cH(n, k, d)$. The families $\cM(n, k, d, r)$ and $\cT(n, k, d, r)$ may, in some sense, give a series of families lying in between $\cA(n, k, d)$ and $\cH(n, k, d)$. As some supporting evidence for this assertion, it may be checked that if $d-1$ divides $k-1$, then
\[\cM\left(n, k, d, \frac{k-1}{d-1}\right) \cong \cT(n, k, d, 1).\]

\end{remark}

\section{Proof of Theorem \ref{thm:extensionthm}}
As we noted in the Introduction, Conjecture \ref{conj:missingoneintconj} is true when $k-\ell$ is a prime power. This is a consequence of a theorem of Frankl and Wilson~\cite{FW}. 

\begin{theorem}[Frankl-Wilson]\label{thm:fwthm}
Let $p$ be a prime and let $q=p^e$ be a prime power. Suppose that $\cF\subset \binom{[n]}{k}$ satisfies $|F\cap F'| \not\equiv k\pmod{q}$ for distinct $F, F' \in \cF$. Then $|\cF| \le \binom{n}{q-1}$. 
\end{theorem}

To prove Theorem \ref{thm:extensionthm}, we use a theorem proven in \cite{KMW}. 
\begin{theorem}\label{thm:ell=1thm}
Suppose that $p$ is prime, $k\in \mathbb{N}$, $L\subset \{0, \ldots, k-1\}$, and $f(x)$ is an integer valued polynomial of degree $d \le k$ such that $f(t) \equiv 0\pmod{p}$ for every $t \in L$ and $f(k) \not\equiv 0 \pmod{p}$. If $\cF$ is a $k$-uniform $L$-intersecting set system on $[n]$, then $|\cF|\le \binom{n}{d}$. 
\end{theorem}

Keevash, Mubayi and Wilson \cite{KMW} used Theorem \ref{thm:ell=1thm} to prove the case $\ell = 1$ of Conjecture \ref{conj:missingoneintconj}. We modify their argument to prove Theorem \ref{thm:extensionthm}.

\begin{proof}[Proof of Theorem \ref{thm:extensionthm}]
We apply Theorem \ref{thm:ell=1thm} with $L = [0, k-1] - \{\ell\}$. Let $f(x) = \binom{x-\ell-1}{k-\ell-1}$, interpreted as a polynomial of degree $k-\ell-1$. Then $f(i) = 0$ for $\ell+1 \le i \le k - 1$ and $f(k) = 1$. For $0 \le i \le \ell - 1$, note that \[f(i) = \binom{i-\ell - 1}{k-\ell - 1} = (-1)^{k-\ell+1}\binom{k-i-1}{\ell - i} = (-1)^{k-\ell+1}\frac{(k-i-1)\cdot \ldots \cdot(k-\ell)}{(\ell - i)!}.\] If $k - \ell$ is not a divisor of $\ell!$, then there is a prime $p$ such that for some $a \ge 1$, $p^a$ divides $k-\ell$, but $p^a$ does not divide $\ell!$. For this $p$, we can see that $f(i) \equiv 0\pmod{p}$ for $0 \le i \le \ell - 1$, so by Theorem \ref{thm:ell=1thm}, $|\cF| \le \binom{n}{k-\ell-1}$. 
\end{proof}

\section{Proof of Theorem \ref{thm:kstintthm}}

Our construction of a $K_{s, t}$-intersecting family is based on the set families 
\[\cA_i(n, d, t) = \{A \subset 2^{[n]}: |A\cap [t+di]| \ge  t+(d-1)i\}.\]
Note that each of these families is $d$-wise, $t$-intersecting. Frankl~\cite{F2} conjectures that the maximum size $d$-wise, $t$-intersecting family on $[n]$ is $\max\{|\cA_i|:\hspace{1mm} 0 \le i \le (n-t)/d\}$. 

\begin{proof}[Proof of Theorem \ref{thm:kstintthm}]
Let $S$ be a (labelled) set of $s$ vertices, and $R$ be a labelled set of $t+2s$ vertices such that $R\cap S = \emptyset$. For each $i \in S$, let $d_R(i)$ be the number of vertices in $R$ adjacent to $i$. Let $\cF$  be the family of all labelled graphs on $n$ vertices such that $d_R(i) \ge t+2s-1$ for every $i \in S$. The family $\cF$ is $K_{s, t}$-intersecting, because if $G_1, G_2 \in \cF$, then the vertices in the  set $S$ in $G_1 \cap G_2$ will have at least $t$ common neighbors from $R$. 

We count the graphs in $\cF$. Each vertex in $S$ has $t+2s+1$ possible neighborhoods in $R$, so there are $(t+2s+1)^s$ possible edge sets on $S\times R$. Each edge of $K_n$ that is not in $S\times R$ (\textit{i.e.} does not have one endpoint in $S$ and the other one in $R$) can independently be or not be in such a graph in $\cF$, so $|\cF| = (t+2s+1)^s2^{\binom{n}{2}-st - 2s^2}$. By assumption, $t > 2^{2s}-2s-1$, so $|\cF| > 2^{\binom{n}{2} - st}$.
\end{proof}

The construction for Theorem~\ref{ks1srtint} is quite similar. For $1\le i\le r$, let $S_i$ be a labelled set of $s_i$ vertices, and $R$ be a labelled set of $t+2\sum_{i}s_i$ vertices such that $S_1, \ldots, S_r$ and $R$ are disjoint. We choose $\cF$ to be the family of all labelled graphs on $n$ vertices such that $d_R(i) \ge t+2\sum_{i}s_i-1$ for every $i \in \bigcup S_j$ and which contain all edges between distinct parts $S_i$ and $S_j$. We omit the rest of the details as they are similar to those in the proof of Theorem~\ref{thm:kstintthm}.

\section*{Acknowledgements}
We thank the anonymous referees for their careful reading of the manuscript and their detailed and helpful comments. We also thank the second referee for pointing out the simple proof of Theorem~\ref{thm:ekrdwisetintthm} which we have adopted. 

After making this manuscript public, it was brought to our attention that  Theorem~\ref{thm:extensionthm} follows from a result of Mubayi and R\"{o}dl~\cite[Theorem 15]{MR}. We thank Dhruv Mubayi for pointing this out to us.


\begin{thebibliography}{}
\bibitem{AK1}
R.~Ahlswede, L.~Khachatrian,\emph{The Complete Intersection Theorem for Systems of Finite Sets}, Europ. J. Combinatorics, {\bf 18} (1997), 125--136.
\bibitem{AK2}
R.~Ahlswede, L.~Khachatrian, \emph{The Complete Nontrivial-Intersection Theorem for Systems of Finite Sets}, J. Combin Theory, Series A, {\bf 76} (1996), 21--38.
\bibitem{AK3}
R.~Ahlswede, L.~Khachatrian, \emph{A pushing-pulling method: new proofs of intersection theorems}, Combinatorica {\bf 19}, (1999), 1--15. 
\bibitem{AS}
N.~Alon, J.~Spencer, \emph{The Probabilistic Method}, 4th edition, Wiley, (2016). 
\bibitem{BZ}
A.~Berger, Y.~Zhao, \emph{$K_4$-intersecting families of graphs}. arXiv:2103.12671. 
\bibitem{BD}
A.~Brace, D.E.~Daykin, \emph{A finite set covering theorem}, Bull. Austral. Math. Soc., {\bf 5} (1971), 197--202. 
\bibitem{C}
D.~Christofides, \emph{A counterexample to a conjecture of Simonovits and S\'os}. Manuscript. 
\bibitem{E}
D.~Ellis, \emph{Intersection Problems in Extremal Combinatorics: Theorems, Techniques and Questions Old and New}, arXiv:2107.06371. 
\bibitem{EFF}
D.~Ellis, Y.~Filmus, E.~Friedgut, \emph{Triangle-intersecting families of graphs}, J. Eur. Math. Soc. {\bf 14} (2012), 841--885. 
\bibitem{EKR}
P.~Erd\H{o}s, C.~Ko, R.~Rado, \emph{Intersection theorems for systems of finite sets}, Quart. J. Math. Oxford Ser. (2) {\bf 12} (1961), 313--320.
\bibitem{F3} 
P.~Frankl, \emph{The Erd\H{o}s-Ko-Rado theorem is true for $n=ckt$}, Combinatorics (Proc. Fifth Hungarian Colloq., Keszthely, 1976), Vol. I, pp. 365–375,
Colloq. Math. Soc. J\'anos Bolyai, 18, North-Holland, Amsterdam-New York, 1978.
\bibitem{F83}
P.~Frankl, \emph{An extremal set theoretical characterization of some Steiner systems}, Combinatorica, {\bf 3}, no. 2, (1983), 193--199. 
\bibitem{F2}
P.~Frankl. \emph{Families of finite sets satisfying a union condition}, Discrete Math. {\bf 26} (1979), 111--118.
\bibitem{F1}
P.~Frankl, \emph{An improved universal bound for $t$-intersecting families}, Eur. J. Combinatorics, {\bf 87}, (2020), Paper 103134, 4 pp. 
\bibitem{F}
P.~Frankl, \emph{Multiply-intersecting families}, J. Combin Theory, Series B, {\bf 53} (1991), 195--234. 
\bibitem{FT}
P.~Frankl, N.~Tokushige, \emph{Extremal Problems for Finite Sets}, Student Mathematical Library, American Mathematical Society, volume 86, (2018). 
\bibitem{FT1}
P.~Frankl, N.~Tokushige, \emph{Invitation to intersection problems for finite sets}, J. Combin. Theory Series A, {\bf 144} (2016), 157--211. 
\bibitem{FT3}
P.~Frankl, N.~Tokushige, \emph{Weighted non-trivial multiply intersecting families}, Combinatorica, {\bf 26}, (2006), 37 -- 46.
\bibitem{FW}
P.~Frankl, R.M.~Wilson, \emph{Intersection theorems with geometric consequences}, Combinatorica {\bf 1}, (1981), no. 4, 357--368.
\bibitem{Fried}
E.~Friedgut, \emph{On the measure of intersecting families, uniqueness and stability}, Combinatorica, {\bf 28}, no. 5, (2008), 503--528.  
\bibitem{G}
Gurobi Optimization, LLC, \emph{Gurobi Optimizer Reference Manual}, (2020), \url{http://www.gurobi.com}.
\bibitem{HK}
J.~Han and Y.~Kohayakawa, \emph{The maximum size of a nontrivial intersecting uniform family that is not a subfamily of the Hilton-Milner family}, Proc. Amer. Math. Soc., {\bf 145}, (2017), 73--87. 
\bibitem{HM}
A.~J.~W. Hilton, E.~C.~Milner, \emph{Some intersection theorems for systems of finite sets}, Quarterly Journal of Mathematics, {\bf 18} (1967), no. 1, 369--384. 
\bibitem{K}
G. O. H.~Katona, \emph{Intersection theorems for systems of finite sets}, Acta Math. Acad. Sci. Hung. {\bf 15} (1964), 329--337. 
\bibitem{KMW}
P.~Keevash, D.~Mubayi, R.~Wilson, \emph{Set systems with no singleton intersection}, SIAM J. Discrete Math, {\bf 20}, 4, (2006), 1031--1041. 
\bibitem{KM}
A.~Kostochka and D.~Mubayi, \emph{The Structure of Large Intersecting Families}, Proc. Amer. Math. Soc., {\bf 145}, no. 6, (2017), 2311--2321. 
\bibitem{MR}
D.~Mubayi and V.~R\"{o}dl, \emph{Specified Intersections}, Trans. Amer. Math. Soc., {\bf 366} (2014), no. 1, 491--504. 
\bibitem{OV}
J.~O'Neill and J.~Verstra\"{e}te, \emph{Non-trivial $d$-wise intersecting families}, J. Combin. Theory, Series A {\bf 178} (2021), Paper 105369, 12pp. 
\bibitem{Tok1}
N.~Tokushige, \emph{Intersecting families -- uniform versus weighted}, Ryukyu Math. Journal, {\bf 18}, (2005), 89--103. 
\bibitem{Tok2}
N.~Tokushige, \emph{The maximum measure of nontrivial $3$-wise intersecting families}, arXiv:2203.17158v1. 
\bibitem{Tok3}
N.~Tokushige, \emph{The maximum size of $3$-wise $t$-intersecting families}, Eur. J. Combinatorics, {\bf 28}, (2007), 152--166. 
\bibitem{T}
N.~Tokushige, \emph{A product version of the Erd\H{o}s-Ko-Rado theorem}, J. Combin. Theory, Series A {\bf 118}, (2011), 1575--1587. 
\bibitem{W}
A.~Zs.~Wagner, \emph{Refuting conjectures in extremal combinatorics via linear programming}, J. Combin. Theory, Series A {\bf 169}, (2020) 105--130. 
\bibitem{Wil}
R.M.~Wilson, \emph{The exact bound in the Erd\H{o}s-Ko-Rado theorem}, Combinatorica, {\bf 4}, (1984), 247--257.

\end{thebibliography}
\end{document}